\documentclass{article}
\usepackage{amsmath, amsthm,graphicx,amsfonts,amssymb,mathrsfs,showidx}
\usepackage{hyperref}
\usepackage{tikz-cd} 
\usepackage{ytableau}
\usepackage{booktabs}

\usepackage{esint}

\usepackage{tikz}
\usetikzlibrary{tqft}

\newtheorem{thm}{Theorem}
\newtheorem{cor}[thm]{Corollary}

\newtheorem{lem}[thm]{Lemma}
\newtheorem{prop}[thm]{Proposition}

\theoremstyle{definition}
\newtheorem{defn}[thm]{\textbf{Definition}}

\theoremstyle{definition}

\theoremstyle{remark}

\definecolor{darkblue}{RGB}{0,0,139}

\hypersetup{
    colorlinks=true,
    linkcolor=darkblue,
    filecolor=darkblue,
    urlcolor=darkblue,
    citecolor=darkblue,
}

\usepackage{relsize}
\usepackage{color}
\usepackage[english]{babel} 

\definecolor{rojo}{rgb}{1,0,0}
\usepackage[all]{xy}
\makeindex
\usepackage{hyperref}
\usepackage{color}
\definecolor{abelian}{cmyk}{0.50,0,1,.4}
\definecolor{noabelian}{cmyk}{0.94,0.54,0,0}
\definecolor{rojo}{cmyk}{0,1,1,0}
\definecolor{verde}{cmyk}{0.91,0,0.88,0.12}
\xyoption{arc}

\title{The Riemann hypothesis through the looking of partitions}
\author{Carlos Segovia\thanks{SECIHTI-UNAM-Oaxaca,  M\'{e}xico. {\em e-mail: }{\tt csegovia@im.unam.mx}}}
\date{}

\begin{document}

\maketitle

\begin{abstract}
An equivalence of the Riemann Hypothesis due to Espinosa reveals a direct bridge to the theory of integer partitions. Building on this formulation, we analyze the asymptotic behavior of a family of combinatorial sums $A_r(n)$, called Espinosa's branches, and their realization by divisors of superabundant numbers. Our main result is an unconditional asymptotic formula for the associated cutoff divisors: for every fixed integer $r\geq 1$, $d^{(r)}(N)\sim e^{-\gamma}(\log N)^{\alpha_r}$, 
as $N$ tends to infinity through superabundant numbers,
where the exponents $\alpha_r$ are given by the Dickman--de Bruijn distribution. This yields explicit sets of divisors $\mathcal{D}_r(N)$ satisfying $\sum_{d\in\mathcal{D}_r(N)}d=\rho_r N\log\log N+o(N)$ for each fixed $r$, where 
$
\rho_r = \lim_{n\to\infty} A_r(n)/n\log\log n
$.
We also compute exactly the contribution of the hook-shaped family of partitions $[r,1^l]$, denoted by $\tilde{\rho}_r$.
Summed over all branches, these contributions account for approximately $91.85\%$ of $e^\gamma\approx 1.781072417990\ldots$. The remaining proportion, coming from all other partitions, satisfies
$
\lim_{r\to\infty} \rho_r/(\rho_r - \tilde{\rho}_r) = 1.
$
\end{abstract}

\section*{Introduction}
In 1859, Riemann, in his seminal memoir \cite{Rie59}, studied the zeta function $\zeta(s)$ of a complex variable and conjectured that all of its nontrivial zeros have real part equal to $1/2$. This statement, the Riemann hypothesis, has since become one of the most important open problems in mathematics. Over time, several striking equivalences have emerged, linking it to the growth of arithmetic functions. Notably, Robin \cite{Ro84} established that the Riemann hypothesis is equivalent to the inequality
\begin{equation}
    \frac{\sigma(n)}{n}<e^\gamma \log\log n,\quad\textrm{ for }n>5040\,,
\end{equation}
with $\sigma(n)$ the sum of divisors and $\gamma$ the Euler–Mascheroni constant. 
Later, Lagarias \cite{Lag02} provided an elementary equivalence of the Riemann hypothesis with the inequality:
\begin{equation}
    \sigma(n)\leq H_n+e^{H_n}\log H_n\,,
\end{equation}
where $H_n$ is the $n$-Harmonic number. Recently, Espinosa \cite{Da25} contributed a new equivalence of the Riemann hypothesis via the inequality: 
    \begin{equation}\label{e1}
    A(n):=\sum_{k=0}^\infty a_k\frac{(H_n)^k}{k!}\geq \sigma(n)\,,\quad\textrm{ for }n>60\,,\end{equation}
    where $-\gamma\leq a_0\leq 0$ and $H_k-\gamma\leq a_k\leq\log k+\frac1k$, for  $k\geq1$.

    Numerous equivalences of the Riemann hypothesis are known; however, the distinctive feature of Espinosa's formulation is that it reveals a concrete connection with integer partitions. 
    For a partition $\mu$ of $j$ with exactly $i$ parts, let $m_\mu(x_1,\cdots, x_n)$ denote the corresponding monomial symmetric polynomial. The Espinosa sums \cite{Da26} are 
\begin{equation}
    E_{i,j}(n) :=  \sum_{\mu \vdash j,\, \ell(\mu) = i} \frac{m_\mu(x_1, \dots, x_n)}{\mu_1! \mu_2! \dots \mu_i!}\,,
\end{equation} 
and the associated $r$-sums, which we call Espinosa's branches, are  
\begin{equation}
   A_r(n)=\sum_{i=1}^n\log(i+r)E_{i,i+r-1} (n)\,. 
\end{equation}
After evaluating at $x_i = 1/i$, these branches recover $A(n)=\sum_{r=1}^\infty A_r(n)$, with the coefficients $a_k=\log(k+1)$, presented by Espinosa in \cite{Da25}.

The case $r=1$ corresponds to a well-understood logarithmic combinatorial structure \cite{ABT00}:
\begin{equation} A_1(n)=\sum_{m=1}^n\log(m+1)e_m(n)\,,\end{equation} 
where the elementary symmetric polynomials, evaluated as $e_m(n):=e_m\left(1, \frac{1}{2}, \ldots, \frac{1}{n}\right)$, admit the probabilistic interpretation $e_m(n)=(n+1)\,\mathbb{P}(C_{n+1}=m+1)$ and $A_1(n)=(n+1)\mathbb{E}[\log C_{n+1}]$. Here $C_n$ denotes the number of cycles in a random permutation of $n$ elements, chosen uniformly. The first asymptotic proportion is given by 
\begin{equation}
    \rho_1:=\lim_{n\rightarrow\infty}\frac{A_1(n)}{n\log\log n}=1\,.
\end{equation}
For $r= 2$, we find $\rho_2:=\lim_{n\to\infty} A_2(n)/n\log\log n=1/2$. To analyze higher values of $r$, we consider a distinguished sub-summand of $A_r(n)$:
\begin{equation}
\tilde{A}_{r}(n) = \frac{1}{r!} \sum_{l=0}^{n-1} \log(l+r+1) \, m_{[r,1^l]}(x_1,x_2,\ldots,x_n)\,.
\end{equation}
which comprises the monomial symmetric polynomials $m_{[r,1^l]}$ associated with the hook-shaped family of partitions $[r]$, $[r,1]$, $[r,1,1]$, $\ldots$ (denoted by $[r,1^l]$ for short). For these families, we obtain the following closed formula for their asymptotic proportions, valid for \(r\geq 2\):
 \begin{equation}
        \tilde{\rho}_r:=\lim_{n\rightarrow\infty} \frac{\tilde{A}_r(n)}{n\log\log n}=\frac{1}{r!}\left[(-1)^r+\sum_{j=3}^r(-1)^{r+j}\zeta(j-1)\right]\,.
    \end{equation}
These proportions yield the following numerical values:

$$
\begin{array}{l|c|c|l}
r & \tilde{\rho}_r & \rho_r-\tilde{\rho}_r & S_r = \sum_{k=1}^r \rho_k \\
\hline
1 & 1 & 0 & 1 \\
2 & \frac{1}{2} & 0 & 1.5 \\
3 & \frac{-1+\zeta(2)}{6}
& \frac{2-\zeta(2)}{4}
& 1.696255494429 \\
4 & \frac{1-\zeta(2)+\zeta(3)}{24}
& \frac{5-\zeta(2)-2\zeta(3)}{24}
& 1.759091951227 \\
5 & \frac{-1+\zeta(2)-\zeta(3)+\zeta(4)}{120}
& \frac{\zeta(2)^2-24\zeta(2)-10\zeta(3)-7\zeta(4)+60}{288}
& 1.776074728193 \\
\vdots & \vdots & \vdots & \vdots \\
14
& 5.735841379278870\times 10^{-12}
& 1.34054104355\times 10^{-9}
& 1.781072417814269 \\
\vdots & \vdots & \vdots & \vdots \\
\infty
& 0
& 0
& e^{\gamma}=1.781072417990\ldots
\end{array}$$

In December 2025, the author posted on
\href{https://mathoverflow.net/questions/506410/is-this-representation-of-e-gamma-in-terms-of-zeta-values-known}
{MathOverflow 506410}
a decomposition of $e^\gamma$ into the asymptotic proportions $\rho_r$.
Recently, G\'omez-A\'iza \cite{GA26} gave a generating function that lets us obtain the exact values of the proportions $\rho_r$. Nevertheless, the hook-shaped family $\sum_{i=1}^\infty\tilde{\rho}_i\approx 1.635904187444$ represents the $91.85 \%$ percentage of $e^\gamma$, we have the asymptotic behavior $\lim_{r\to\infty} \rho_r/(\rho_r - \tilde{\rho}_r) = 1$.

The author in \cite{Seg26} proposed the realization of Espinosa's branches, where each branch is associated with a specific subset of divisors. The underlying idea is that the Riemann hypothesis is equivalent to a progressive realization of Espinosa's branches: Let $S_r = \sum_{i=1}^r \rho_i$ and for a superabundant number $N>5040$ write
$$
1=d_1<d_2<\cdots<d_{\tau(N)}=N.
$$
The $r$-th cutoff index is given by
$$
j_r(N):=
\max\left\{
j\in\{1,\cdots,\tau(N)\}:\sum_{i=1}^j\frac{1}{d_i}
\leq S_r\log\log N
\right\},
$$
and the corresponding cutoff divisor by
$$
d^{(r)}(N):=d_{j_r(N)}.
$$
We say that $N$ is $r$-saturated if $d^{(r)}(N)=N$, or equivalently $j_r(N)=\tau(N)$. 

The $r$-th Espinosa branch is said to be progressively realizable if, for every superabundant number \(N>5040\) such that $d^{(r)}(N)<N$, there exists an integer $s=s(N,r)>r$ such that $d^{(s)}(N)>d^{(r)}(N)$.
In Proposition \ref{prop:progressive-Robin} we show that every Espinosa branch is progressively realizable if and only if Robin's inequality holds for every superabundant number $N>5040$.
In Figure \ref{fig1}, we plot the values of
$
S_r \log\log N - \sigma(N)/N
$
for the first 10,000 terms of the computed colossally abundant sequence (assuming the Alaoglu-Erd\"os conjecture), displayed on a logarithmic scale. 
Positive values in Figure \ref{fig1} correspond to saturation, whereas negative values correspond to nonsaturation.

\begin{figure}
    \centering
    \includegraphics[width=0.85\linewidth]{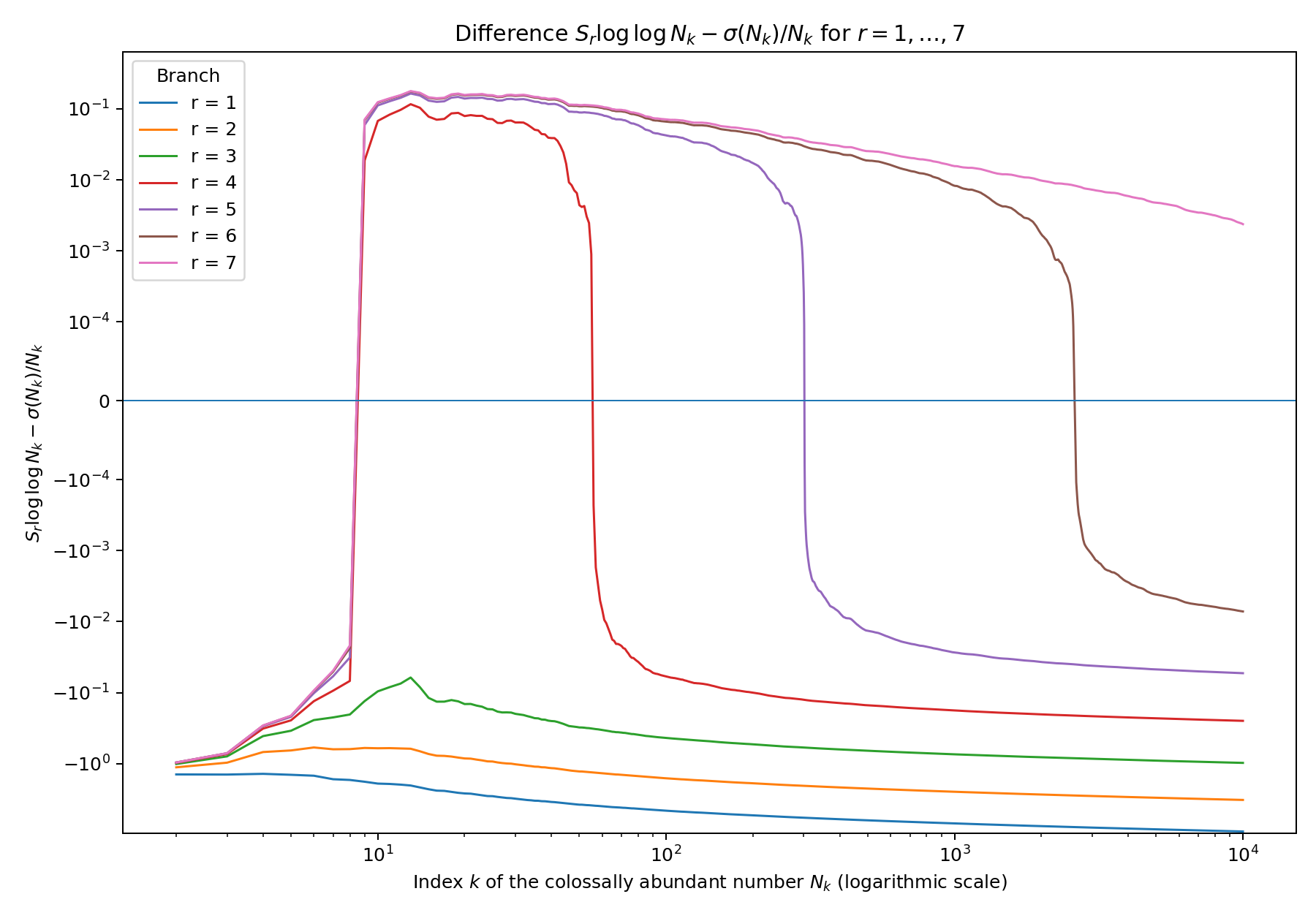}
    \caption{Saturation margins for the computed colossally abundant numbers and the first seven Espinosa branches.}
    \label{fig1}
\end{figure}

Our main result, Theorem \ref{thmpr}, is the asymptotic expression for the cutoff divisors of superabundant numbers $N$:
$$
\lim_{N\to\infty} \frac{d^{(r)}(N)}{e^{-\gamma}(\log N)^{\alpha_r}} = 1,
$$
where the exponents $\alpha_r$ are defined by the Dickman--de Bruijn distribution \cite{Di30,Br66}. These numbers, which we call the Dickman--de Bruijn exponents, reveal how the Espinosa branches model the behaviour of the divisors of superabundant numbers at infinity.


In retrospect, this is consistent with progressive realizability: for superabundant numbers and each $r \geq 1$, we obtain
$$
\lim_{N\to\infty} \frac{d^{(r+1)}(N)}{d^{(r)}(N)} = \infty,
$$
which means that, for every positive integer $r\geq 1$, there exists a threshold $N_r$ such that $d^{(r+1)}(N) > d^{(r)}(N)$ for every superabundant number $N > N_r$. However, this implication for all the Espinosa branches does not imply Robin's inequality for all superabundant numbers. A control of the saturation ranges for each Espinosa branch is still required. Nevertheless, this provides an explicit subset of divisors 
$$\mathcal{D}_r(N)=\left\{\frac{N}{d_j}:j_{r-1}(N)<j\leq j_r(N)\right\}\,,$$
with $j_0(N)=0$, realizing the $r$-th Espinosa branch.
For every fixed integer $r\geq 1$, as $N\rightarrow\infty$ through superabundant numbers, 
\begin{equation}
    \sum_{d\in \mathcal{D}_r(N)}d=\rho_rN\log\log N+o(N)\,.
\end{equation}
In particular, $\sum_{d\in \mathcal{D}_r(N)}d\sim A_r(N)$.

\section{Espinosa equivalence of the Riemann hypothesis}  
Robin-Lagarias-Espinosa type equivalences \cite{Ro84,Lag02,Da25} remarkably encode the Riemann Hypothesis. 
However, Espinosa’s equivalence reveals a connection to integer partitions.

\begin{thm}[Espinosa \cite{Da25}]
The Riemann hypothesis is equivalent to the inequality
\begin{equation}\label{e1}
A(n) := \sum_{k=0}^\infty a_k \frac{(H_n)^k}{k!} \geq \sigma(n), \qquad \text{for } n > 60,
\end{equation}
where $-\gamma\leq a_0\leq 0$ and $H_k - \gamma \leq a_k \leq \log k + \frac{1}{k}$ for $k \geq 1$.
\end{thm}
\begin{proof}
   First, take with the coefficients $a_k=H_k-\gamma$ and we show the following identity:
   \begin{equation}\label{grrrr}
    \sum_{k=0}^\infty\left(\frac{H_k-\gamma}{k!}\right)x^k=e^x(\log x+E_1(x))\,,
\end{equation}
where $E_1(x) = \int_x^\infty e^{-t} \, dt/t$ for $x > 0$. This identity is the reason why the Espinosa equivalence can be understood as a reformulation of the Robin--Lagarias equivalences.

Using the integral representation $H_k=\int_{0}^1(1-t^k)/(1-t)dt$, we obtain  
\begin{align}
    \sum_{k=1}^\infty H_k\frac{x^k}{k!}&=\int_0^1\frac{\sum_{k=1}^\infty x^k/k!-\sum_{k=1}^\infty (xt)^k/k!}{1-t}dt\\
    &=\int_0^1\frac{e^x-1-(e^{xt}-1)}{1-t}dt\\
    &=\int_0^1\frac{e^x-e^{xt}}{1-t}dt \\
    &=e^x\int_0^1\frac{1-e^{-x(1-t)}}{1-t}dt.
\end{align}
Applying the change of variable $u=1-t$, so that $dt=-du$, yields
\begin{align}
     e^x\int_0^1\frac{1-e^{-x(1-t)}}{1-t}dt&=e^x\int_0^1\frac{1-e^{-xu}}{u}du,
\end{align}
and with the further substitution $t=xu$, hence $du=dt/x$, we get
\begin{align}
    e^x\int_0^1\frac{1-e^{-xu}}{u}du&=e^x\int_0^x\frac{1-e^{-t}}{t}dt.
\end{align}
According to \cite[5.1.39]{AS66}, the following identity holds:
\begin{equation}
    \int_0^x\frac{1-e^{-t}}{t}dt=\log x+\gamma+E_1(x)\,.
\end{equation}
We finally obtain
\begin{align}
    \sum_{k=0}^\infty\left(\frac{H_k-\gamma}{k!}\right)x^k &=\sum_{k=1}^\infty H_k\frac{x^k}{k!}-\gamma e^x\\
    &=e^x(\log x+\gamma+E_1(x))-\gamma e^x\\
    &=e^x(\log x+E_1(x))\,.\label{eq1}
\end{align}

Assuming the Riemann hypothesis, we consider two results: by Lagarias \cite[Lem. 3.1]{Lag02}, for $n\geq 3$, we have
\begin{equation}\label{lag0}e^\gamma n\log\log n\leq e^{H_n}\log H_n\,.\end{equation}
and the Robin's criterion \cite{Ro84} implies that for every $n > 5040$, we have the inequality
\begin{equation}
\sigma(n)=\sum_{d|n}d\leq e^\gamma n\log\log n\leq  e^{H_n}(\log H_n+E_1(H_n))\leq A(n)   
\end{equation}
For the remaining values $60< n\leq 5040$, a direct computer verification confirms that $\sigma(n)\leq A(n)$ holds as well.

Now, suppose that the inequality $\sigma(n) \leq A(n)$ holds for $n>60$. 
From our previous discussion for the coefficients $a_k=H_k-\gamma$ we have $A(n)=e^{H_n}\bigl(\log H_n+E_1(H_n)\bigr)$.

If the Riemann hypothesis were false, the result of Robin \cite{Ro84} guarantees the existence of constants $0 < \beta < \frac{1}{2}$ and $C > 0$ such that
\begin{equation}
    e^\gamma n\log\log n + \frac{C n\log\log n}{(\log n)^\beta} \leq \sigma(n)
\end{equation}
holds for infinitely many $n$ (which can be chosen arbitrarily large). 

On the other hand, by Lagarias \cite[p. 541]{Lag02}, for $n \geq 3$, we have the inequality
\begin{equation}\label{lag1}
    e^{H_n}\log H_n\leq e^\gamma n\log\log n+\frac{3}{2}\frac{e^\gamma n}{\log n}\,.
\end{equation}
Consequently, we obtain  
\begin{align}
    \frac{Cn\log \log n}{(\log n)^\beta}&\leq \frac{3}{2}\frac{e^\gamma n}{\log n}+e^{H_n}E_1(H_n)\\
    &\leq \frac{3}{2}\frac{e^\gamma n}{\log n}+\frac{e^\gamma(n+1)}{\log n}\\
    &=\frac{e^\gamma}{\log n}\left(\frac{5}{2}n+1\right)
\end{align}
where we have used $e^{H_n}E_1(H_n)\leq e^{H_n}/H_n\leq e^\gamma(n+1)/\log n$, which follows from the well-known bounds $\log n < H_n < \log(n+1)+\gamma$ for $n\geq 1$.

From this, we deduce the inequality  
\begin{equation}
\frac{C}{e^\gamma}\log\log(n)\leq (\log(n))^{\beta-1}\left(\frac{5}{2}+\frac{1}{n}\right)\,.
\end{equation}
Taking the limit as $n\to\infty$, the left-hand side tends to infinity, whereas the right-hand side tends to zero, a contradiction.

We may choose coefficients $a_k$ satisfying $-\gamma\leq a_0\leq 0$ and $H_k-\gamma \leq a_k \leq \log k + 1/k$,  with $k\geq 1$, providing the following inequality is satisfied(see de Appendix):
\begin{equation}\label{apen1}
    A(n)\leq e^{H_n}\log H_n+\frac{2e^\gamma (n+1)}{\log n}\,.
\end{equation} 
Indeed, by reproducing the previous calculations, we arrive at the inequality
\begin{equation}
\frac{C}{e^\gamma}\log\log(n)\leq (\log(n))^{\beta-1}\left(\frac{7}{2}+\frac{2}{n}\right)\,.
\end{equation}
The same limiting argument gives a contradiction, proving the equivalence for the general coefficients.
\end{proof}

\begin{prop}
We have the following asymptotic limit:
\begin{equation}
    \lim_{n\rightarrow \infty}\frac{A(n)}{n\log\log n}=e^\gamma.
\end{equation}
\end{prop}
\begin{proof}
     From \eqref{apen1} and \eqref{lag1} we obtain, for $n\geq 3$,
    \begin{equation}
        \frac{A(n)}{n\log \log n}\leq e^\gamma+\frac{e^\gamma}{\log n\log\log n}\left(\frac{7}{2}+\frac{2}{n}\right)\,.
    \end{equation}
    Taking the limit as $n\to\infty$ yields
 $\lim_{n\rightarrow\infty}A(n)/n\log\log n\leq e^\gamma$. 
 
 To establish the reverse inequality, we use \eqref{grrrr} and \eqref{lag0}. For the coefficients $a_k=H_k-\gamma$, we obtain, for $n\geq 3$, 
    \begin{equation}
        \frac{A(n)}{n\log\log n}=\frac{e^{H_n}\log H_n+e^{H_n}E_1(H_n)}{n\log\log n}\geq e^\gamma+\frac{e^{H_n}E_1(H_n)}{n\log\log n}\,.
    \end{equation}
    where \begin{equation}\lim_{n\to\infty}\frac{e^{H_n}E_1(H_n)}{n\log\log n}=0\,.\end{equation} 
    This completes the proof. 
\end{proof}

\section{Espinosa's branches and the contribution of hook partitions}

 In this section our approach is purely algebraic: we work with monomial symmetric polynomials, and the whole computational engine relies on Newton's identities for symmetric polynomials. 

Following the discussion in the previous section, we now specialize the sum $A(n)$ by choosing the coefficients $a_k = \log(k+1)$. That is, we consider the expression 
\begin{equation}
    A(n):=\sum_{k=0}^\infty \log(k+1)\frac{(H_n)^k}{k!}\,.
\end{equation}

We consider the Espinosa sums \cite{Da26}. For a partition $\mu$ of $j$ with $\ell(\mu)=i$ parts, let $m_\mu(x_1,\ldots,x_n)$ denote the corresponding monomial symmetric polynomial.  We then define
\begin{equation}
    E_{i,j}(n) :=  \sum_{\mu \vdash j,\, \ell(\mu) = i} \frac{m_\mu(x_1, \dots, x_n)}{\mu_1! \mu_2! \dots \mu_i!}\,.
\end{equation}
Building on these, we introduce the \(r\)-sums, which we call the \(r\)-th Espinosa branch:
\begin{equation}
   A_r(n)=\sum_{i=1}^n\log(i+r)E_{i,i+r-1} (n)\,, 
\end{equation}
and by evaluating the symmetric polynomials at $x_i = 1/i$ for $i=1,\ldots,n$, we recover the original sum as $A(n) = \sum_{r=1}^{\infty} A_r(n)$.

We will need two classical families of symmetric functions: the elementary symmetric polynomials $e_k$ and the power sums $p_k$, defined respectively by

\begin{equation}
e_k = \sum_{1 \leq i_1 < \cdots < i_k \leq n} x_{i_1} \cdots x_{i_k}, \qquad
p_k = \sum_{i=1}^n x_i^k\,,
\end{equation}
where $e_0=1$ and $e_k=0$ for $k>n$.

The case $r=1$ is particularly well understood and corresponds to a classical logarithmic combinatorial structure \cite{ABT00}. In this case, the sum reduces to elementary symmetric polynomials:
\begin{equation} A_1(n)=\sum_{m=1}^n\log(m+1)e_m(n)\,,\end{equation} 
where $e_m(n) := e_m(1, 1/2, \ldots, 1/n)$. These polynomials admit an explicit expression in terms of Stirling numbers of the first kind: \begin{equation}e_m(n):=e_m\left(1, \frac{1}{2}, \ldots, \frac{1}{n}\right) = \genfrac{[}{]}{0pt}{}{n+1}{m+1}/n!\,,\end{equation} 
Moreover, they admit a probabilistic interpretation \cite{ABT00}: if $C_n$ denotes the number of cycles in a uniformly random permutation of $n$ elements, then
\begin{equation} e_m(n)=(n+1)\mathbb{P}(C_{n+1}=m+1)\,,\end{equation}
where $\mathbb{P}(C_{n+1}=m+1)$ denotes the probability that a permutation of $n+1$ elements, chosen uniformly at random, has exactly $m+1$ cycles. As a consequence, we obtain 
    \begin{equation}\label{equaprimera-segunda}
A_1(n) = (n+1) \sum_{k=1}^{n+1} \log k \cdot \mathbb{P}(C_{n+1}=k) = (n+1) \cdot \mathbb{E}[\log C_{n+1}]\,.        
    \end{equation}

For $r \geq 2$, we consider the sub-summand of hook partitions $[r, 1^l]$, i.e., one part equal to $r$ and $l$ parts equal to $1$, and define
\begin{equation}
\tilde{A}_{r}(n) = \frac{1}{r!} \sum_{l=0}^{n-1} \log(l+r+1) \, m_{[r,1^l]}(x_1,x_2,\ldots,x_n)\,.
\end{equation}
For $r=1$ and $r=2$, these sub-sums actually account for the entirety of $A_1(n)$ and $A_2(n)$. 

The elementary symmetric polynomials $e_k$ and the power sums $p_k$ are related by the Newton identities, of which the following determinantal formula due to Macdonald \cite[pg.~28]{Mac95} is an immediate consequence:

\begin{equation}\label{MacDonald}
p_k = 
\begin{vmatrix}
e_1 & 1 & 0 & 0 & \cdots & 0 \\
2e_2 & e_1 & 1 & 0 & \cdots & 0 \\
3e_3 & e_2 & e_1 & 1 & \cdots & 0 \\
\vdots & \vdots & \vdots & \vdots & \ddots & \vdots \\
ke_k & e_{k-1} & e_{k-2} & \cdots & e_2 & e_1
\end{vmatrix}\,.\end{equation}
We have a determinantal formula for the monomial symmetric polynomials $m_{[r,1^l]}$.
\begin{prop}\label{prop:determinant-m} For $r\geq 2$ and $l\geq 0$, we have the following determinantal formula:
\begin{equation}
    m_{[r,1^l]}=\begin{vmatrix}
e_1 & 1 & 0 & 0 & \cdots & 0 \\
2e_2 & e_1 & 1 & 0 & \cdots & 0 \\
3e_3 & e_2 & e_1 & 1 & \cdots & 0 \\
\vdots & \vdots & \vdots & \vdots & \ddots & \vdots \\
(r-1)e_{r-1} & e_{r-2} & e_{r-3} & \cdots & e_1 & 1 \\
(l+r)e_{l+r} & e_{l+r-1} & e_{l+r-2} & \cdots & e_{l+2} & e_{l+1}
\end{vmatrix}\,.
\end{equation}
\end{prop}
\begin{proof}The case $l=0$ is exactly the determinantal formula \eqref{MacDonald}. For the general case, consider the product
\begin{equation}
p_r \cdot e_{l+1} = \left( \sum_{j=1}^{n} x_j^r \right) \left( \sum_{1 \leq i_1 < i_2 < \cdots < i_{l+1} \leq n} x_{i_1} x_{i_2} \cdots x_{i_{l+1}} \right)    
\end{equation}
Expanding the product, we separate the terms according to whether the index $j$ appears among the $i_k$ or not:
\begin{equation}
p_r \cdot e_{l+1} = 
\sum_{\substack{1 \leq j \leq n \\ 1 \leq i_1 < \cdots < i_{l+1} \leq n \\ j = i_k \text{ for some } k}} 
x_j^r \cdot x_{i_1} \cdots x_{i_{l+1}}
\;+\;
\sum_{\substack{1 \leq j \leq n \\ 1 \leq i_1 < \cdots < i_{l+1} \leq n \\ j \neq i_k \text{ for all } k}} 
x_j^r \cdot x_{i_1} \cdots x_{i_{l+1}}.
\end{equation}
The first sum collects the terms where $j$ coincides with one of the indices $i_k$. These terms correspond to the monomial symmetric polynomial associated with the partition $[r+1, 1^l]$, namely $m_{[r+1, 1^l]}$. The second sum collects the terms where $j$ is different from all $i_k$. This corresponds to the monomial symmetric polynomial associated with the partition $[r, 1^{l+1}]$, namely $m_{[r, 1^{l+1}]}$. Hence, we obtain the relation \begin{equation}m_{[r, 1^{l+1}]} = p_r \cdot e_{l+1} - m_{[r+1, 1^l]}\,.\end{equation}
Using this recursion together with the determinantal formula for $p_r$, the result follows by simultaneous induction on $l$, valid for all $r\geq 2$.
\end{proof}

From Proposition \ref{prop:determinant-m}, we obtain a determinantal expression for $\tilde{A}_r(n)$ as follows.

\begin{prop}\label{prop:determinant-tildeA}
For $r\geq 2$, we have
\begin{equation}
\tilde{A}_r(n) = \frac{1}{r!} \det 
\begin{pmatrix}
e_1(n) & 1 & 0 & \cdots & 0 \\
2e_2(n) & e_1(n) & 1 & \cdots & 0 \\
\vdots & \vdots & \vdots & \ddots & \vdots \\
(r-1)e_{r-1}(n) & e_{r-2}(n) & \cdots & e_1 & 1 \\
T_1^{(r)} & T_2^{(r)} & \cdots & T_{r-1}^{(r)} & T_r^{(r)}
\end{pmatrix},
\end{equation}
where
\begin{equation}
T_1^{(r)} = \sum_{l=0}^{n-1} \log(l+r+1) \cdot (l+r) e_{l+r}(n),
\end{equation}
and for $j\in\{2,\ldots,r\}$,
\begin{equation}
T_j^{(r)} = \sum_{l=0}^{n-1} \log(l+r+1) \cdot e_{l+r-j+1}(n).
\end{equation}
\end{prop}
\begin{defn}\label{numexp}For $C := C_{n+1}$, we define the following numerical expectations 
\begin{equation}
E_1 := \mathbb{E}[(C-1)\log C], \qquad 
E_j := \mathbb{E}[\log(C+j-1)] \;\; (j\geq 2).
\end{equation}
\end{defn}
Using these definitions, we can rewrite the sums $T_j^{(r)}$ as follows:
\begin{align}
T_1^{(r)} &= (n+1)E_1 - \sum_{t=1}^{r} \log t \cdot (t-1) \, e_{t-1}(n),\\
T_j^{(r)} &= (n+1)E_j - \sum_{t=1}^{r-j+1} \log(t+j-1) \, e_{t-1}(n), \quad j\geq 2.
\end{align}

A key observation is that the following determinant is zero 
\begin{equation}
\operatorname{det}
\begin{pmatrix}
e_1 & 1 & 0 & 0 & \cdots & 0 \\
2e_2 & e_1 & 1 & 0 & \cdots & 0 \\
3e_3 & e_2 & e_1 & 1 & \cdots & 0 \\
\vdots & \vdots & \vdots & \vdots & \ddots & \vdots \\
(r-1)e_{r-1} & e_{r-2} & e_{r-3} & \cdots & e_1 & 1 \\
\Delta_1 & \Delta_2 & \Delta_3 & \cdots & \Delta_{r-1} & \Delta_r
\end{pmatrix}=0\,,
\end{equation}
where 
\begin{equation}
\Delta_1 := -\sum_{t=1}^{r} \log t \cdot (t-1) e_{t-1}, \qquad 
\Delta_j := -\sum_{t=1}^{r-j+1} \log(t+j-1) e_{t-1} \;\; (j\geq 2).
\end{equation}
Indeed, if $R_l$ denotes the $l$-th of the first $r-1$ rows, then $(\Delta_1,\cdots,\Delta_r)=-\sum_{l=1}^{r-1}\log(l+1)R_l$, so the last row is a linear combination of the preceding rows.

An immediate consequence is the following.

\begin{cor}\label{cor:tildeA-explicit}
For $r\geq 2$, we have
\begin{equation}
\tilde{A}_r(n) = \frac{n+1}{r!}\left((-1)^{r+1}E_1+ \sum_{j=2}^r (-1)^{r+j} p_{j-1}(n) \, E_j\,\right),
\end{equation}
where $p_k(n) = \sum_{i=1}^n 1/i^k$, $k\geq 1$.
\end{cor}

For $C = C_{n+1}$, and every fixed $a\geq 0$, in Proposition \ref{cor:E-asymp} of the Appendix we show the following asymptotic expansion:
\begin{equation}
\mathbb{E}[\log(C+a)] = \log(H_{n+1}+a) + O\left(\frac{1}{\log n}\right),
\end{equation}
and
\begin{equation}
\mathbb{E}[C\log C] = H_{n+1}\log H_{n+1} + \frac{\kappa_2}{2H_{n+1}} + O\left(\frac{1}{\log n}\right),
\end{equation}
where $\kappa_2 = H_{n+1} - p_2(n+1)$ is the second cumulant.

Thus we can find asymptotic expressions for the numerical expectations in Definition \ref{numexp}.

\begin{prop}\label{prop:E-asymp-explicit}
For $C = C_{n+1}$, the following asymptotic expansions hold:
\begin{align}
E_1 &= \log n \log\log n + (\gamma-1)\log\log n + \gamma + O(1), \label{E1-exp}\\
E_j &= \log\log n + O\left(\frac{1}{\log n}\right), \qquad j\geq 2. \label{Ej-exp}
\end{align}
Consequently, the following relation holds:
\begin{equation}
H_n E_2 - E_1 = \log\log n - \gamma + O(1). \label{E-relation}
\end{equation}
\end{prop}
\begin{proof}
We begin with the expression for $E_1$:
\begin{align}
E_1 &= \mathbb{E}[C\log C] - \mathbb{E}[\log C] \nonumber\\
&= \left[H_{n+1}\log H_{n+1} + \frac{\kappa_2}{2H_{n+1}} + O\left(\frac{1}{\log n}\right)\right] 
   - \left[\log H_{n+1} + O\left(\frac{1}{\log n}\right)\right] \nonumber\\
&= \log H_{n+1}(H_{n+1}-1) + \frac{\kappa_2}{2H_{n+1}} + O\left(\frac{1}{\log n}\right). \nonumber
\end{align}
Proposition \ref{equahar} in the Appendix provides the following asymptotic expansions:
    \begin{equation}
        H_{n+1}=\log n+\gamma+O\left(\frac{1}{n}\right)\textrm{ and }\log H_{n+1}=\log\log n+\frac{\gamma}{\log n}+o\left(\frac{1}{\log n}\right)\,.
    \end{equation}
and $\kappa_2 = O(H_{n+1})$. In consequence, we have the following 
\begin{align}
E_1 &= \left(\log\log n + \frac{\gamma}{\log n} + o\left(\frac{1}{\log n}\right)\right)
      \left(\log n + \gamma - 1 + O\left(\frac{1}{n}\right)\right) + O(1) \nonumber\\
&= \log n \log\log n + (\gamma-1)\log\log n + \gamma + O(1),
\end{align}
which proves \eqref{E1-exp}.

For $E_j$ with $j\geq 2$, we use the Taylor expansion
$\log(H_{n+1}+j-1) = \log H_{n+1} + (j-1)/H_{n+1} + O(1/H_{n+1}^2)$,
valid since $j$ is fixed and $H_{n+1}\to\infty$. This gives
\begin{align}
E_j &= \mathbb{E}[\log(C+j-1)] \nonumber\\
&= \left[\log H_{n+1} + \frac{j-1}{H_{n+1}} + O\left(\frac{1}{H_{n+1}^2}\right)\right] + O\left(\frac{1}{\log n}\right) \nonumber\\
&= \log\log n + O\left(\frac{1}{\log n}\right),
\end{align}
which proves \eqref{Ej-exp}.

Finally, using \eqref{E1-exp} and \eqref{Ej-exp} with $j=2$, we obtain
\begin{align}
H_n E_2 - E_1
&= H_n\left(\log\log n + O\left(\frac{1}{\log n}\right)\right)
   - \left(\log n \log\log n + (\gamma-1)\log\log n + \gamma + O(1)\right) \nonumber\\
&= (H_n - \log n)\log\log n - (\gamma-1)\log\log n - \gamma + O(1) \nonumber\\
&= (\gamma + O(1/n))\log\log n - (\gamma-1)\log\log n - \gamma + O(1) \nonumber\\
&= \log\log n - \gamma + O(1),
\end{align}
where we have used $H_n = \log n + \gamma + O(1/n)$. This establishes \eqref{E-relation}.
\end{proof}
We are now in a position to compute the contribution of the hook partitions. 
\begin{thm}\label{thm:main}
Let $C = C_{n+1}$ and define the asymptotic proportions
\begin{equation}
\rho_1 := \lim_{n\to\infty} \frac{A_1(n)}{n\log\log n}, \qquad 
\tilde{\rho}_r := \lim_{n\to\infty} \frac{\tilde{A}_r(n)}{n\log\log n}\ \ (r\geq 2).
\end{equation}
Then $\rho_1 = 1$, and for $r\geq 2$,
\begin{equation}
 \tilde{\rho}_r = \frac{1}{r!}\left[(-1)^r + \sum_{j=3}^r (-1)^{r+j}\,\zeta(j-1)\right]\,.
\end{equation}
\end{thm}

\begin{proof}
The case $r=1$ follows directly from the probabilistic interpretation $A_1(n) = (n+1)\,\mathbb{E}[\log C_{n+1}]$ where
\begin{equation}
\mathbb{E}[\log C_{n+1}] = \log H_{n+1} + O(1/\log n)
= \log\log n + O(1/\log n)\,,\end{equation}
Consequently,
\begin{equation}
\lim_{n\to\infty} \frac{A_1(n)}{n\log\log n}
= \lim_{n\to\infty} \frac{n+1}{n}\left(1+O\left(\frac{1}{\log n\log\log n}\right)
   \right)
= 1.    
\end{equation}
Thus $\rho_1 = 1$, as claimed.

For $r\geq 2$, we have $p_1(n)=H_n$, and for $k\geq 2$, $p_k(n) = \zeta(k) + O(1/n^{k-1})$.
Corollary \ref{cor:tildeA-explicit} implies 
\begin{equation}
\tilde{A}_r(n)  = \frac{n+1}{r!} \Big[ (-1)^{r+1}(E_1 - H_n E_2) + \sum_{j=3}^r (-1)^{r+j} p_{j-1}(n) E_j \Big]
\end{equation}
and recall the relation $H_n E_2 - E_1 = \log\log n - \gamma + O(1)$. Hence
\begin{equation}
(-1)^{r+1}(E_1 - H_n E_2) = (-1)^r (\log\log n - \gamma) + O(1).
\end{equation}
For $j\geq 3$, using $p_{j-1}(n) = \zeta(j-1) + O(1/n^{j-2})$ and $E_j = \log\log n + O(1/\log n)$, we obtain
\begin{equation}
p_{j-1}(n) E_j = \zeta(j-1)\log\log n + O(1).
\end{equation}
Substituting into $\tilde{A}_r(n)$:
\begin{equation}
\tilde{A}_r(n) = \frac{n+1}{r!} \Big[ (-1)^r (\log\log n - \gamma) + \sum_{j=3}^r (-1)^{r+j} \zeta(j-1)\log\log n + O(1) \Big].
\end{equation}
In consequence, 
\begin{equation}
\begin{aligned}
\lim_{n\rightarrow\infty}\frac{\tilde{A}_r(n)}{n\log\log n} &= \lim_{n\rightarrow\infty}\frac{(1+\frac{1}{n})}{r!} \left[ (-1)^r -\frac{(-1)^r\gamma}{\log\log n} + \sum_{j=3}^r (-1)^{r+j} \zeta(j-1) + O\left(\frac{1}{\log\log n}\right)\right] \\&=\frac{1}{r!}\left[(-1)^r + \sum_{j=3}^r (-1)^{r+j}\zeta(j-1)\right].
\end{aligned}
\end{equation}
\end{proof}

\subsection{Generating functions and the non-hook contribution}

Recently, G\'omez-A\'iza \cite{GA26} confirmed the contribution of the hook partitions obtained by the author in \cite{Seg26}. His formulation starts from the finite product
\begin{equation}
F_n(u,v)=\prod_{q=1}^n\left(1+v(e^{u/q}-1)\right)\,,
\qquad
E_{i,j}(n)=[u^jv^i]F_n(u,v)\,.    
\end{equation}
After normalizing the coefficients appearing in $A_r(n)$, the logarithmic
weight can be interpreted as an expectation, giving
\begin{equation}
A_r(n)
=
\left(\sum_{i\geq0}E_{i,i+r-1}(n)\right)
\mathbb E_{n,r}[\log(I_{n,r}+r)]\,,
\end{equation}
for some random variables $I_{n,r}$. G\'omez-A\'iza proves in \cite[Thm. 3.5]{GA26} that
\begin{equation}
\sum_{i\geq0}E_{i,i+r-1}(n)\sim \rho_r n\,,
\qquad
\mathbb E_{n,r}[\log(I_{n,r}+r)]\sim\log\log n\,,
\end{equation}
and hence
\begin{equation}\label{lplp}
A_r(n)\sim \rho_r n\log\log n\,.
\end{equation}
The complete proportions have the generating function \cite[eq. (11)]{GA26}
\begin{equation}
    H(z)=\sum_{r=1}^\infty \rho_rz^{r-1}=\prod_{m=1}^\infty \frac{m}{m+1}\left(1+\frac{e^{z/m}-1}{z}\right)\,.
\end{equation}
In particular, 
\begin{equation}\label{wert}
    \sum_{r=1}^\infty\rho_r=H(1)=\lim_{n\rightarrow\infty}\prod_{m=1}^n\frac{m}{m+1}e^{1/m}=\lim_{n\rightarrow\infty}\frac{e^{H_n}}{n+1}=e^\gamma\,.
\end{equation}
For the hook contribution, G\'omez-A\'iza proves in \cite[Cor. 4.3]{GA26} that
\begin{equation}
 \tilde{\rho}_r=\frac1{r!}\sum_{m=1}^\infty\frac1{m^{r-1}(m+1)},
 \qquad r\ge2,
\end{equation}
and $\tilde{\rho}_1=\rho_1=1$. Since all summands are nonnegative, summing first over $r$ yields
\begin{align}
 \sum_{r=1}^\infty\tilde{\rho}_r
 &=1+\sum_{m\ge1}\frac{m}{m+1}
         \sum_{r=2}^\infty\frac{1}{r!m^r}\\
 &=1+\sum_{m\ge1}\frac{m}{m+1}
         \left(e^{1/m}-1-\frac1m\right)
 =1.635904187444\ldots .
\end{align}
This is approximately $91.85\%$ of $e^\gamma$. Nevertheless, the non-hook contribution is asymptotically dominant.

\begin{prop} For $r\ge3$, we have $\rho_r-\tilde{\rho}_r>0$ and
\begin{equation}\label{rty}
 0\le\frac{\tilde{\rho}_r}{\rho_r-\tilde{\rho}_r}\le\frac{24}{r}.
\end{equation}
Consequently,
\begin{equation}
 \lim_{r\to\infty}\frac{\rho_r}{\rho_r-\tilde{\rho}_r}=1+\lim_{r\to\infty}\frac{\tilde{\rho}_r}{\rho_r-\tilde{\rho}_r}=1.
\end{equation}
\end{prop}
\begin{proof}
The explicit formula for the hook contribution gives
\begin{equation}\label{2345}
    0<\tilde\rho_r
    =\frac{1}{r!}\sum_{m=1}^{\infty}\frac{1}{m^{r-1}(m+1)}
    \leq\frac{1}{r!}\sum_{m=1}^{\infty}\frac{1}{m(m+1)}
    =\frac{1}{r!}\,.
\end{equation}
The non-hook contribution is a sum of positive contributions associated with Young shapes having
at least two parts greater than $1$; see \cite[Def.~4.1 and Thm.~4.2]{GA26}. Therefore, keeping only those with exactly two such parts, we obtain 
\begin{equation}
    \rho_r-\tilde\rho_r\geq
    \sum_{1\leq m<n}\ \sum_{k=1}^{r-2}
    \frac{1}{(k+1)!(r-k)!\,m^k(m+1)\,n^{r-1-k}(n+1)}\,.
\end{equation}
Since all summands are positive, taking $m=1$, $n=2$, and $k=r-2$, we obtain for $r\geq3$,
\begin{equation}\label{5678}
    \rho_r-\tilde\rho_r
    \geq\frac{1}{(r-1)!\,2!\,(1+1)\,2(2+1)}
    =\frac{1}{24(r-1)!}>0\,.
\end{equation}
From \eqref{2345} and \eqref{5678} we obtain \eqref{rty}.
\end{proof}

\section{Realization of Espinosa's branches by divisors}\label{section3}
The author in \cite{Seg26} proposed the realization of Espinosa's branches, where each branch is associated with a specific subset of divisors. Notwithstanding, to have an equivalence with the Riemann Hypothesis we require a progressively realizability of Espinosa branches as we show in Proposition \ref{prop:progressive-Robin}.

\begin{defn}
Let $S_r = \sum_{i=1}^r \rho_i$,
and for a superabundant number $N>5040$ write
\begin{equation}
1=d_1<d_2<\cdots<d_{\tau(N)}=N\,.    
\end{equation}
The $r$-th cutoff index is given by
\begin{equation}
j_r(N):=
\max\left\{
j\in\{1,\cdots,\tau(N)\}:\sum_{i=1}^j\frac{1}{d_i}
\leq S_r\log\log N
\right\}\,,
\end{equation}
and the corresponding cutoff divisor by
\begin{equation}
    d^{(r)}(N):=d_{j_r(N)}\,.
\end{equation}
We say that $N$ is $r$-saturated if $d^{(r)}(N)=N$, 
or equivalently $j_r(N)=\tau(N)$. 
\end{defn}
Notice we always, for $N>5040$, have the sequences $j_1(N)\leq j_2(N)\leq \cdots\leq \tau(N)$ and  $d^{(1)}(N)\leq d^{(2)}(N)\leq \cdots\leq N$ and if $N$ is $r$-saturated, then $N$ is $s$-saturated for all $s\geq r$.

\begin{defn}[Progressive Realizability]
The $r$-th Espinosa branch is said to be progressively realizable
if, for every superabundant number \(N>5040\) such that
$d^{(r)}(N)<N$, there exists an integer $s=s(N,r)>r$ such that
\begin{equation}
d^{(s)}(N)>d^{(r)}(N)\,.    
\end{equation}
\end{defn}

\begin{prop}\label{prop:progressive-Robin}
Every Espinosa branch is progressively realizable if and only if
Robin's inequality holds for every superabundant number $N>5040$.
\end{prop}

\begin{proof}
We first recall that $S_r=\sum_{j=1}^{r}\rho_j$, where $\lim_{r\rightarrow\infty}S_r=e^\gamma$; see \eqref{wert}.
First, assume that every Espinosa branch is progressively realizable. For every superabundant number $N>5040$, the sequence
\begin{equation}
d^{(1)}(N) \le d^{(2)}(N) \le \cdots \le N    
\end{equation}
eventually stabilizes, progressive realizability implies that $d^{(R)}(N)=N$ for some $R$. Consequently,
\begin{equation}
\sigma(N) \le S_R N \log\log N < e^\gamma N \log\log N\,.    
\end{equation}
If Robin's inequality fails for some integer $n>5040$, its least counterexample is superabundant. The preceding inequality excludes this possibility. 

Conversely, suppose that Robin’s inequality holds for every superabundant number $N>5040$.
Fix such an $N$ and an integer $r\ge1$ with $d^{(r)}(N)<N$. Since $S_r$ increase to $e^\gamma$, there exists an integer $R>r$ such that
\begin{equation}
 \sigma(N)<S_R N \log\log N\,.    
\end{equation}
Thus $d^{(R)}(N)=N$. Consequently, there exists a first index $s$ with $r<s\le R$ such that
\begin{equation}
d^{(s)}(N) > d^{(r)}(N)\,.    
\end{equation}
\end{proof}

\begin{lem}[Progression to the next branch] Let $N > 5040$ be a superabundant number\footnote{This result remains valid for any integer $\geq 5040$ upon extending the definition of the cutoff divisors.} and let $r \ge 1$. Suppose there exists an integer $m \ge 1$ such that
\begin{equation}
H_m \le S_r \log\log N, \qquad \frac{1}{m+1} \le \rho_{r+1} \log\log N\,,     
\end{equation}
Then, if $d^{(r)}(N) < N$, it follows that $d^{(r+1)}(N) > d^{(r)}(N)$.
\end{lem}

\begin{proof}
Let $j := j_r(N) < \tau(N)$. We claim first that $m < \tau(N)$. Indeed, if $m\geq \tau(N) $, then
\begin{equation}
\sum_{i=1}^{\tau(N)} \frac{1}{d_i} \le H_{\tau(N)} \le H_m \le S_r \log\log N\,,
\end{equation}
so $j_r(N) = \tau(N)$, which means $d^{(r)}(N) = N$, a contradiction.

By maximality of $j$ we have $j \ge m$, and therefore $m+1 \le j+1 \le d_{j+1}$. Hence
\begin{equation}
\frac{1}{d_{j+1}} \le \frac{1}{m+1} \le \rho_{r+1} \log\log N\,.
\end{equation}
Consequently,
\begin{equation}
\sum_{i=1}^{j+1} \frac{1}{d_i}
= \sum_{i=1}^{j} \frac{1}{d_i} + \frac{1}{d_{j+1}}
\le (S_r + \rho_{r+1}) \log\log N
= S_{r+1} \log\log N\,.
\end{equation}
Thus $j+1$ is admissible for level $r+1$, so $j_{r+1}(N) \ge j+1$. Therefore
\begin{equation}
d^{(r+1)}(N) = d_{j_{r+1}(N)} \ge d_{j+1} > d_j = d^{(r)}(N)\,. 
\end{equation}
\end{proof}

\begin{prop}
For all superabundant number $N>5040$ and each $r\in \{1,2,3,4\}$ we have that if $d^{(r)}(N)<N$, then $d^{(r+1)}(N)>d^{(r)}(N)$.
In particular, the first four branches are progressively realizable. 
\end{prop}
\begin{proof}
We take the lowest superabundant bigger than 5040 which is 10080. Thus any superabundant $N>5040$ satisfies 
\begin{equation}\log\log N\geq \log\log 10080>\frac{111}{50}\,.\end{equation} 

For case $r=1$, we take $m=1$ where $S_1=1$ and $\rho_2=1/2$ and we obtain 
\begin{equation}
    H_1=1<\log\log N,\hspace{1cm}\frac{1}{m+1}=\frac{1}{2}<\frac{1}{2}\log\log N\,.
\end{equation}

For case $r=2$, we take $m=15$ where $S_2=3/2$ and we obtain 
\begin{equation}H_{15}=3.318228993...<\frac{333}{100}=\frac{3}{2}\frac{111}{50}<S_2\log\log N\,,\hspace{1cm}\frac{1}{16}<\frac{259}{600}=\frac{7}{36}\frac{111}{50}<\rho_3\log\log N\,.\end{equation}

For case $r=3$, we take $m=23$ and we obtain 
\begin{equation}H_{23}=3.734291511...<\frac{2257}{600}=\frac{61}{36}\frac{111}{50}<S_3\log\log N\,,\hspace{0.5cm}\frac{1}{24}<\frac{629}{4800}=\frac{17}{288}\frac{111}{50}<\rho_4\log\log N\,.\end{equation}

For case $r=4$, we take $m=27$ and we obtain 
\begin{equation}H_{27}=3.891456753...<\frac{3737}{960}=\frac{505}{288}\frac{111}{50}<S_4\log\log N,\hspace{0.5cm}\frac{1}{28}<\frac{37}{1000}=\frac{1}{60}\frac{111}{50}<\rho_5\log\log N\,.\end{equation}
\end{proof}

We conjecture that each Espinosa branch can be progressively realized at the next level, namely: for all superabundant number $N>5040$ if $d^{(r)}(N)<N$, then $d^{(r+1)}(N)>d^{(r)}(N)$.

\section{The Dickman--de Bruijn exponents}
\label{Secto}

In this section we denote by $N$ a superabundant number and denote by $P=P(N)$ its largest prime factor with the convention $P(1)=1$. 
We also consider the following sums:

\begin{equation}
     F(x,y):=\sum_{\substack{m\leq x \\ P(m)\leq y}} \frac{1}{m}\,,\textrm{ and } F_N(x):=\sum_{\substack{d|N \\ d\leq x}} \frac{1}{d}\,.
\end{equation}

\begin{defn}[Dickman function \cite{Di30,Br66}] Consider the function $\varrho:[0,\infty)\rightarrow\mathbb{R}$, the unique continuous function satisfying $\varrho(u)=1$ for $0\leq u\leq 1$ and the differential equation
\begin{equation}\label{dif-equ}u\varrho'(u)=-\varrho(u-1)\,, \qquad u > 1,\end{equation}
with $\varrho(u)=0$ for $u<0$. 
\end{defn}
The construction is defined successively on each interval $[m,m+1]$, for $m\geq 1$, by 
\begin{equation}\varrho(u)=\varrho(m)-\int_m^u\frac{\varrho(t-1)}{t}dt\,.\end{equation}
The Dickman function  satisfies $\varrho(u)>0$ for $u\geq 0$. Moreover, it is strictly decreasing on $(1,\infty)$, with
\begin{equation}
    \int_0^\infty\varrho(u)du=e^\gamma\,.
\end{equation}
The Dickman--de Bruijn distribution is $e^{-\gamma}J(u)$ where
\begin{equation}
    J(u):=\int_0^u\varrho(t)dt\,.
\end{equation}

\begin{defn}
    For each integer $r \geq 1$, we define the \emph{$r$-th Dickman--de Bruijn exponent} $\alpha_r$ as the unique solution of
    \begin{equation}
        J(\alpha_r) = S_r,
    \end{equation}
    that is,
    \begin{equation}
        \int_0^{\alpha_r} \varrho(t)\,dt = \sum_{j=1}^r \rho_j.
    \end{equation}
\end{defn}

\begin{thm}\label{thmpr}
    For every fixed integer $r \geq 1$ and for superabundant numbers $N$,
    \begin{equation}
        \lim_{N \to \infty} \frac{d^{(r)}(N)}{e^{-\gamma}(\log N)^{\alpha_r}} = 1,
    \end{equation}
    or equivalently,
    \begin{equation}
        d^{(r)}(N) \sim e^{-\gamma}(\log N)^{\alpha_r}.
    \end{equation}
\end{thm}

For this result, we need some technical lemmas.

\begin{lem}
  For fixed $0 < a < U < \infty$ and $y$ sufficiently large, we have the asymptotic expression
\begin{equation}
    F(y^u, y) = (\log y)J(u) + \gamma \varrho(u) + O_{a,U}\!\left(\frac{1}{\log y}\right),
\end{equation}
uniformly in $u \in [a,U]$.
\end{lem}

\begin{proof}
This estimate is the cumulative form of a formula established by Tenenbaum \cite[Lemma~2.1, equation~(2.2)]{Te06} for the reciprocal tail of friable integers. Define
    \begin{equation}
        T(y,z) := \sum_{\substack{P(n) \le y \\ n > z}} \frac{1}{n}.
    \end{equation}
    Tenenbaum's estimate states that, for $v = \log z / \log y$ and $\mathcal{T}(v) := \int_v^\infty \varrho(t)\,dt$,
    \begin{equation}\label{propT}
        T(y,z) = \mathcal{T}(v)\log y - \gamma \varrho(v) + O\!\left(\frac{\varrho(v)\log(v+2)}{\log y} + \frac{1}{z}\right)\,,
    \end{equation}
valid for $y\geq 2$ and $1\leq z\leq \exp\exp\left((\log y)^{3/5-\varepsilon}\right)$, where we take $0<\varepsilon<3/5$, and $\log z\geq \log\log  y$.

    We use the exact identity
    \begin{equation}\label{cond1234}
        F(y^u, y) = H_{\lfloor y \rfloor} + T(y,y) - T(y,y^u)\,,
    \end{equation}
    
    Applying \eqref{propT} with $z = y$ and $z=y^u$, we obtain
    \begin{align}
        T(y,y) &= \mathcal{T}(1)\log y - \gamma + O\!\left(\frac{1}{\log y}\right), \label{eqpo1}\\
        T(y,y^u) &= \mathcal{T}(u)\log y - \gamma \varrho(u) + O_{a,U}\!\left(\frac{1}{\log y}\right)\,.\label{equados}
    \end{align}
    In \eqref{equados} we use that $\varrho(u)\log(u+2)$ is bounded on $[a,U]$ and that $1/y^u \le y^{-a} = O_a(1/\log y)$. 

    The conditions of \eqref{propT} are satisfied in these two particular cases for $y$ sufficiently large, taking $\varepsilon =1/10$, it suffices to choose $y$ large enough so that
    \begin{equation}
        \frac{\log\log y}{\log y}\leq \min\{a,1\}\,,\hspace{1cm}\max\{U,1\}\leq \frac{\exp(\sqrt{\log y})}{\log y}\,.
    \end{equation}

    Using $H_{\lfloor y \rfloor} = \log y + \gamma + O(1/y)$ and substituting \eqref{eqpo1} and \eqref{equados} in \eqref{cond1234}, we get
    \begin{equation}
        F(y^u, y) = \bigl(1 + \mathcal{T}(1) - \mathcal{T}(u)\bigr)\log y + \gamma \varrho(u) + O_{a,U}\!\left(\frac{1}{\log y}\right).
    \end{equation}
    Finally, since $\mathcal{T}(1) - \mathcal{T}(u) = \int_1^u \varrho(t)\,dt$ and $\varrho(t) = 1$ on $[0,1]$,
    \begin{equation}
        1 + \mathcal{T}(1) - \mathcal{T}(u) = \int_0^1 \varrho(t)\,dt + \int_1^u \varrho(t)\,dt = \int_0^u \varrho(t)\,dt = J(u).
    \end{equation}
    Therefore,
    \begin{equation}
        F(y^u, y) = (\log y)J(u) + \gamma \varrho(u) + O_{a,U}\!\left(\frac{1}{\log y}\right).
    \end{equation}
\end{proof}

\begin{lem}
    For $N$ a superabundant number and $P = P(N)$ the largest prime. For every fixed $U > 0$,
    \begin{equation}
        0 \le F(x, P) - F_N(x) \ll_U \frac{\log P}{\sqrt{P}} = o(1),
    \end{equation}
    uniformly for $1 \le x \le P^U$, as $N \to \infty$.
\end{lem}

\begin{proof}
    The assumptions imply that $F_N(x) \le F(x, P)$. For a positive integer $m$ with $m \le x$, $P(m)\leq P$ and $m\nmid N$, we must have $q^{a_q+1} \mid m$ for some prime $q \le P$, where $a_q = v_q(N)$ is the maximum exponent of $q$ in $N$. Hence,
    \begin{equation}\label{ladodere}
        F(x, P) - F_N(x)\leq 
        \sum_{q\leq P}\sum_{\substack{m\leq x\\q^{a_q+1}|m}}\frac{1}{m}= \sum_{q \le P} \frac{1}{q^{a_q+1}} H_{\lfloor x/q^{a_q+1} \rfloor}
        \le (1 + \log x) \sum_{q \le P} \frac{1}{q^{a_q+1}}.
    \end{equation}
We use the standard fact that, for a superabundant number $N$ with largest prime factor $P$, the exponents satisfy
\begin{equation}
    q^{a_q+1} > P \quad \text{for every } q < P.
\end{equation}
To prove this, assume that $q^{a_q+1} < P$ (the case $q^{a_q+1} = P$ is excluded since $q < P$) and construct the integer
\begin{equation}
    M := \frac{N\,q^{a_q+1}}{P},
\end{equation}
for which we have $M < N$. For a positive integer $n$, denote by
\begin{equation}
    h(n) = \frac{\sigma(n)}{n}
\end{equation}
which is a multiplicative function.

Since $N$ is a superabundant number, we have $h(M) < h(N)$. Moreover, for $a=a_q$ and $b=a_P$ 
\begin{equation}
    \frac{h(M)}{h(N)}
    = \frac{h(q^{2a+1})\,h(P^{b-1})}{h(q^{a})\,h(P^{b})}
    \geq  \left(1 + \frac{1}{q^{a+1}}\right)\left(\frac{h(P^{b-1})}{h(P^{b})}\right)
    \geq \frac{1 + 1/q^{a+1}}{1 + 1/P}
    > 1,
\end{equation}
which contradicts the assumption that $N$ is a superabundant number.

Now, we split the sum on the right-hand side of \eqref{ladodere} at $\sqrt{P}$, which gives
    \begin{equation}\label{dado1}
        \sum_{q \le \sqrt{P}} \frac{1}{q^{a_q+1}}\leq P^{-1/2}\textrm{ and }\sum_{\sqrt{P}<q\leq P}\frac{1}{q^{a_q+1}}\leq \sum_{n>\sqrt{P}}\frac{1}{n^2} \ll P^{-1/2}\,,
    \end{equation}
    where in the first sum of \eqref{dado1} we use that the inequality $q^{a_q+1} > P$ holds for at most $\sqrt{P}$ primes. 
    
    For $1\leq x\leq P^U$, we have $1+\log x\le 1+U\log P\le \left(U+\frac{1}{\log 2}\right)\log P$.
Since $(\log P)/\sqrt P\to0$, the asserted uniform convergence follows.
\end{proof}

\begin{prop}\label{propim}
    Fix $s\in (0,e^\gamma)$ and $c>0$, and let $\alpha>0$ be the uniquely solution of $s=J(\alpha)$. Then, for all sufficiently large superabundant numbers $N$, the following asymptotic formula holds:
    \begin{equation}
        F_N(c(\log N)^\alpha)=s\log\log N+\varrho(\alpha)(\log c+\gamma)+o(1)\,.
    \end{equation}
\end{prop}

\begin{proof}
    Fix $0<a<\alpha<U$ (for instance, $a=\alpha/2$ and $U=\max(2,\alpha+1)$ where we require $U>1$). For $u\in [a,U]$, the preceding lemmas yield, uniformly on $[a,U]$,
    \begin{equation}
        F_N(P^u)=\log P\,J(u)+\gamma\varrho(u)+O_{a,U}\left(\frac{1}{\log P}+\frac{\log P}{\sqrt{P}}\right).
    \end{equation}
    We evaluate $F_N$ at $P^{u_N}=c(\log N)^\alpha$, where
    \begin{equation}\label{utili}
        u_N=\alpha+\frac{\log c+\alpha\delta_N}{\log P},
    \end{equation}
    with $\delta_N=\log\log N-\log P\rightarrow 0$, since $\log N\sim P$; see Alaoglu--Erd\"os \cite[Theorem~7]{AlEr44}. Consequently, $u_N\rightarrow \alpha$, and hence $u_N\in [a,U]$ for all sufficiently large $N$.

    Since $\varrho$ is Lipschitz and $J'=\varrho$, setting $h:=(\log c+\alpha\delta_N)/\log P$, we have
    \begin{equation}\label{dido12}
        J(\alpha+h)=J(\alpha)+\varrho(\alpha)h+O_\alpha(h^2),
        \quad\text{and}\quad
        \varrho(\alpha+h)=\varrho(\alpha)+O_\alpha(|h|).
    \end{equation}
    Moreover, since $h=O_{\alpha,c}(1/\log P)$, we use \eqref{utili}-\eqref{dido12} and we obtain
    \begin{align}
    &\log P\,J(u_N)=s\log P+\varrho(\alpha)(\log c+\alpha\delta_N)+O_{\alpha,c}(1/\log P),\\
    &\gamma\varrho(u_N)=\gamma\varrho(\alpha)+O_{\alpha,c}(1/\log P).
\end{align}
    where recall $s=J(\alpha)$.

As a consequence, we obtain 
    \begin{equation}
        F_N(c(\log N)^\alpha)
        = s\log P + s\delta_N + \varrho(\alpha)(\log c + \gamma) + (\alpha\varrho(\alpha) - s)\delta_N
        + O\left(\frac{1}{\log P} + \frac{\log P}{\sqrt{P}}\right),
    \end{equation}
    where $s\log P+s\delta_N=s\log\log N$ and $(\alpha\varrho(\alpha)-s)\delta_N=o(1)$, since $\delta_N\rightarrow 0$. Finally, as $\log P/\sqrt{P}\rightarrow 0$, the desired asymptotic formula follows.
\end{proof}

\begin{proof}[Proof of Theorem \ref{thmpr}]
    Fix a positive integer $r\geq 1$, write $s=S_r$ and $\alpha=\alpha_r$. Fix   
    $\epsilon >0$ and take $t\in \{\epsilon,-\epsilon/2,-\epsilon\}$. We apply Proposition \ref{propim} with $X_N=e^{-\gamma}(\log N)^\alpha$, 
    \begin{equation}F_N(e^t X_N)=s\log\log N+t\varrho(\alpha)+o(1)\,.
\end{equation}
Since $\varrho(\alpha)>0$, for large $N$ we have $F_N(e^\epsilon X_N)>s\log\log N$. By the definition of $d^{(r)}(N)$ we have $d^{(r)}(N)<e^\epsilon X_N$.

For the lower bound, for large $N$, we have 
\begin{equation}
    F_N(e^{-\epsilon/2}X_N)=s\log\log N-\frac{\epsilon}{2}\varrho(\alpha)+o(1)<s\log\log N\,.
\end{equation} 
Additionally, 
\begin{equation}F_N(e^{-\epsilon/2}X_N)-F_N(e^{-\epsilon}X_N)=\frac{\epsilon}{2}\varrho(\alpha)+o(1)>0\,.\end{equation}
Therefore, there is an admissible divisor in the interval $(e^{-\epsilon}X_N,e^{-\epsilon/2}X_N]$. In consequence, we have the inequality 
$$e^{-\epsilon}<\frac{d^{(r)}(N)}{X_N}<e^\epsilon$$
which shows the theorem.
\end{proof}

As an application of the present section we have the following corollaries.

\begin{cor}
    For each integer $r\geq 1$ fixed, we have 
    \begin{equation}
        \frac{d^{(r+1)}(N)}{d^{(r)}(N)}\sim(\log N)^{\alpha_{r+1}-\alpha_r}\rightarrow\infty\,,
    \end{equation}
when $N\rightarrow\infty$ for superabundant numbers $N$. In particular, there exists a threshold $N_r$ such that 
\begin{equation}d^{(r+1)}(N)>d^{(r)}(N)\,,\end{equation}
for every superabundant number $N>N_r$. 
\end{cor}


\begin{cor}
Put $j_0(N)=0$ and 
\begin{equation}
\mathcal{D}_r(N)=\left\{\frac{N}{d_j}:j_{r-1}(N)<j\leq j_r(N)\right\}\,.
\end{equation}
For every fixed integer $r\geq 1$, as $N\rightarrow\infty$ through superabundant numbers, 
\begin{equation}\label{fghfj}
    \sum_{d\in \mathcal{D}_r(N)}d=\rho_rN\log\log N+o(N)\,.
\end{equation}
In particular, $\sum_{d\in \mathcal{D}_r(N)}d\sim A_r(N)$.
\end{cor}

\begin{proof}
Fix $r\geq1$. For each $1\leq k\leq r$, Theorem~17 gives
$d^{(k)}(N)\to\infty$ and eventually $d^{(k)}(N)<N$ for all sufficiently
large superabundant $N$.  Define $\varepsilon_k(N)
:=S_k\log\log N-\sum_{i=1}^{j_k(N)}\frac1{d_i}$. The maximality of $j_k(N)$ implies
\begin{equation}
0\leq\varepsilon_k(N)
<\frac1{d_{j_k(N)+1}}
\leq\frac1{d^{(k)}(N)}
=o(1).
\end{equation}
Set $S_0=0$ and $\varepsilon_0(N)=0$, and interpret the sum
at level zero as zero. Since $S_r-S_{r-1}=\rho_r$, we obtain
\begin{equation}
\sum_{d\in D_r(N)}d=N\sum_{i=j_{r-1}(N)+1}^{j_r(N)}\frac1{d_i}=\rho_rN\log\log N
  +N\bigl(\varepsilon_{r-1}(N)-\varepsilon_r(N)\bigr)
\end{equation}
where $\varepsilon_{r-1}(N)-\varepsilon_r(N)\rightarrow 0$ which shows \eqref{fghfj}. The final assertion follows from \eqref{lplp}.
\end{proof}

\section*{Appendix}
\label{append}
Consider the exponential integrals $E_1(x) = \int_x^\infty e^{-t}/tdt$, $\operatorname{Ei}(x) = \fint^x_{-\infty} e^{t}/t,dt$, $x>0$.

\begin{lem}\label{lem:coeff-bound}
For the choice of coefficients $a_k = \log k + \frac{1}{k}$, with $k \geq 1$, $a_0=0$, and for any $n > 1$, we have the inequality
\begin{equation}
\sum_{k=1}^\infty \left(\log k + \frac{1}{k}\right)\frac{(H_n)^k}{k!}
\leq e^{H_n}\log H_n + \frac{2e^\gamma (n+1)}{\log n}.
\end{equation}
\end{lem}
\begin{proof}
    We have the following sequence of identities and inequalities:
    \begin{align}
    \sum_{k=1}^\infty \left(\log k+\frac{1}{k}\right)\frac{(H_n)^k}{k!}&=\sum_{k=1}^\infty \log k\frac{(H_n)^k}{k!}+\sum_{k=1}^\infty \frac{(H_n)^k}{k\cdot k!}\,,\\
    &<\sum_{k=1}^\infty (H_k-\gamma)\frac{(H_n)^k}{k!}+\operatorname{Ei}(H_n)-\gamma-\log H_n\,,\\
    &= e^{H_n}\log H_n+e^{H_n}E_1(H_n)+\operatorname{Ei}(H_n)-\log H_n\,.
\end{align}
where we used \cite[5.1.10]{AS66} together with $\log(k)<H_k-\gamma$, for $k\geq 1$, and the identity \eqref{grrrr}. 
Consequently, it remains to show
\begin{equation}\label{eq2}
e^{H_n}E_1(H_n)+\operatorname{Ei}(H_n)\leq\log(H_n)+\frac{2e^\gamma(n+1)}{\log(n)}\,.
\end{equation}
We have the well-known inequality $e^xE_1(x)\leq 1/x$, and one can show that for $x>2$, the inequality $\operatorname{Ei}(x)\leq 2e^x/x$ holds. In addition, $e^{H_n}/H_n\leq e^\gamma (n+1)/\log n$ follows from the classical bounds $\log n < H_n < \log(n+1)+\gamma$ for $n\geq 1$. Moreover, $\log H_n > 1/\log n$ for $n>3$.

Putting these estimates together, we obtain
\[
e^{H_n}E_1(H_n)+\operatorname{Ei}(H_n) \leq \frac{1}{H_n} + \frac{2e^{H_n}}{H_n}
\leq \frac{1}{\log n} + \frac{2e^\gamma(n+1)}{\log n}
\leq \log H_n + \frac{2e^\gamma(n+1)}{\log n},
\]
which establishes \eqref{eq2} for $H_n > 2$. The remaining cases where $H_n \leq 2$ (i.e., small $n$) can be verified directly by numerical computation.
\end{proof}

\begin{prop}\label{equahar}
We have the following asymptotic expressions:
    \begin{equation}
        H_{n+1}=\log n+\gamma+O\left(\frac{1}{n}\right)\textrm{ and }\log H_{n+1}=\log\log n+\frac{\gamma}{\log n}+o\left(\frac{1}{\log n}\right)\,.
    \end{equation}
\end{prop}
\begin{proof}
    We use the expression for the harmonic number 
    \begin{equation}
        H_{n+1}=\log (n+1) +\gamma+\frac{1}{2(n+1)}-\frac{1}{12(n+1)^2}+\cdots 
    \end{equation}
Using the expansions $\log(n+1) = \log n + \log(1 + 1/n) = \log n + \frac{1}{n} - \frac{1}{2n^2} + O\left(\frac{1}{n^3}\right)$ and $\frac{1}{n+1} = \frac{1}{n} - \frac{1}{n^2} + O\left(\frac{1}{n^3}\right)$, we substitute into the expression above. After simplification, we obtain
\begin{equation}
H_{n+1} = \log n + \gamma + O\left(\frac{1}{n}\right),
\end{equation}
which establishes the first formula.

For the second, we take logarithms. Write $H_{n+1} = \log n + \gamma + \varepsilon_n$, where $\varepsilon_n = O(1/n)$. Then
\begin{equation}
\log H_{n+1} = \log\bigl(\log n + \gamma + \varepsilon_n\bigr) = \log(\log n + \gamma) + \log\left(1 + \frac{\varepsilon_n}{\log n + \gamma}\right).
\end{equation}

Set $L_n = \log n$ and $A_n = L_n + \gamma$. Note that $\varepsilon_n/A_n = O(1/(n \log n)) \to 0$ as $n \to \infty$. Using the expansion $\log(1+x) = x + O(x^2)$ for $x \to 0$, we obtain
\begin{equation}
\log H_{n+1} = \log A_n + \frac{\varepsilon_n}{A_n} + O\left(\frac{\varepsilon_n^2}{A_n^2}\right).
\end{equation}

Now, $\log A_n = \log(L_n + \gamma) = \log L_n + \frac{\gamma}{L_n} + O\left(\frac{1}{L_n^2}\right)$. Moreover, $\varepsilon_n/A_n = O(1/(n \log n)) = o(1/L_n)$, and the error term satisfies $O(\varepsilon_n^2/A_n^2) = O(1/(n^2 \log^2 n)) = o(1/L_n)$. Putting everything together,
\begin{align}
\log H_{n+1} &= \log L_n + \frac{\gamma}{L_n} + O\left(\frac{1}{L_n^2}\right) + \frac{\varepsilon_n}{A_n} + o\left(\frac{1}{L_n}\right) \\
&= \log\log n + \frac{\gamma}{\log n} + o\left(\frac{1}{\log n}\right),
\end{align}
as desired.
\end{proof}

\begin{prop}\label{cor:E-asymp}
For $C = C_{n+1}$, for every fixed $a\geq 0$, we have the following
\begin{equation}
\mathbb{E}[\log(C+a)] = \log(H_{n+1}+a) + O\left(\frac{1}{\log n}\right),
\end{equation}
and
\begin{equation}
\mathbb{E}[C\log C] = H_{n+1}\log H_{n+1} + \frac{\kappa_2}{2H_{n+1}} + O\left(\frac{1}{\log n}\right),
\end{equation}
where $\kappa_2 = H_{n+1} - p_2(n+1)$ is the second cumulant.
\end{prop}

\begin{proof}
    The central moments $\mu_m := \mathbb{E}[(C-H_{n+1})^m]$ are expressed in terms of the cumulants $\kappa_i$ of the number of cycles $C_{n+1}$:
\begin{equation}\label{ewo1}
\begin{array}{c}
\kappa_1 = H_{n+1} = p_1(n+1),\\
\kappa_2 = p_1(n+1) - p_2(n+1),\\
\kappa_3 = p_1(n+1) - 3p_2(n+1) + 2p_3(n+1),\\
\vdots\\
\kappa_m = \displaystyle\sum_{\ell=1}^m (-1)^{\ell-1} (\ell-1)! \, S(m,\ell) \, p_\ell(n+1),
\end{array}
\end{equation}
where $S(m,\ell)$ denotes the Stirling numbers of the second kind. Additionally, the central moments are given via the complete Bell polynomials:
\begin{equation}
\mu_m = B_m(0,\kappa_2,\ldots,\kappa_m).
\end{equation}
The first few central moments are
\begin{equation}\label{ewo2}
\mu_1=0,\quad \mu_2=\kappa_2,\quad \mu_3=\kappa_3,\quad \mu_4=\kappa_4+3\kappa_2^2,\quad \mu_5=\kappa_5+10\kappa_3\kappa_2,\quad \mu_6=\kappa_6+15\kappa_4\kappa_2+10\kappa_3^2+15\kappa_2^3,\ldots
\end{equation}

The number of cycles can be written as a sum of independent Bernoulli variables $C = B_1 + \cdots + B_{n+1}$, with $\mathbb{P}(B_k = 1) = 1/k$ for $1 \leq k \leq n+1$.

To avoid logarithmic series and the use of expectations of infinite sums, we focus on the event $G=\{C \geq \mu/2\}$, where $\mu = H_{n+1}$, and set $X := C - \mu$. From \eqref{ewo1} and \eqref{ewo2}, we have
$$
\mathbb{E}[X] = 0, \qquad \mathbb{E}[X^2] = \kappa_2 = O(\mu), \qquad \mathbb{E}[X^3] = O(\mu), \qquad \mathbb{E}[X^4] = O(\mu^2).
$$

For $f(x) := \log(x+a)$, we apply a finite Taylor expansion:
$$
f(C) = f(\mu) + f'(\mu)(C-\mu) + \frac{1}{2} f''(\xi)(C-\mu)^2,
$$
which can be rewritten as
\begin{equation}\label{re1}
\log(C+a) = \log(\mu+a) + \frac{X}{\mu+a} + R,
\end{equation}
where
$$
R = -\frac{X^2}{2(\xi+a)^2}.
$$

On the event $G$, both $C$ and $\mu$ are at least $\mu/2$.
Since $\xi$ lies between $C$ and $\mu$, we have $\xi\geq\mu/2$, and therefore
$$
\mathbb{E}\left[|R| \mathbf{1}_{G}\right]
\leq \frac{2}{\mu^2} \mathbb{E}[(C-\mu)^2]
= \frac{1}{\mu^2} O(\mu)
= O(1/\mu).
$$

On the complementary event $G^c=\{C < \mu/2\}$ (with $C \geq 1$), the Chernoff bound gives
$$
\mathbb{P}(C < \mu/2) \leq e^{-\mu/8}.
$$
Using the bound
$$
|R| \leq |\log(C+a)| + |\log(\mu+a)| + \frac{|C-\mu|}{\mu+a},
$$
we note that $|\log(C+a)| = O(\log \mu)$, $|\log(\mu+a)| = O(\log \mu)$, and since $|C-\mu| \leq \mu$, the last term is $O(1)$. Hence $|R| = O(\log \mu)$, and consequently
$$
\mathbb{E}\left[|R| \mathbf{1}_{G^c}\right]
= O(\log \mu) \, \mathbb{P}(C < \mu/2)
= O(e^{-\mu/8} \log \mu)
= o(1/\mu).
$$
Combining both cases, we obtain $|\mathbb{E}[R]|\leq
\mathbb{E}\left[|R|\mathbf{1}_G\right]+
\mathbb{E}\left[|R|\mathbf{1}_{G^c}\right]=
O\left(1/\mu\right)$. Taking expectations in \eqref{re1} yields the first equation of the Proposition.

For the second equation, we set $g(x) = x \log x$ and use the finite Taylor expansion
$$
g(C) = g(\mu) + g'(\mu)X + \frac{g''(\mu)}{2}X^2 + \frac{g'''(\mu)}{6}X^3 + \frac{g^{(4)}(\xi)}{24}X^4,
$$
which is equivalent to
\begin{equation}\label{re2}
C \log C = \mu \log \mu + (\log \mu + 1)X + \frac{X^2}{2\mu} - \frac{X^3}{6\mu^2} + R_4,
\end{equation}
where
$$
R_4 = \frac{X^4}{12 \xi^3}.
$$
On the event $G$, both $C$ and $\mu$ are at least $\mu/2$.
Since $\xi$ lies between $C$ and $\mu$, we have $\xi\geq\mu/2$, and therefore
$$
\mathbb{E}\left[|R_4|\mathbf{1}_{G}\right]
\leq \frac{16}{\mu^3} \mathbb{E}[X^4]
= O(1/\mu).
$$

On the complementary event $G^c=\{C < \mu/2\}$, we use the bound
$$
|R_4| \leq |C \log C| + |\mu \log \mu| +(\log \mu + 1)|X| +\frac{X^2}{2\mu} + \frac{|X|^3}{6\mu^2}.
$$
Since $|C \log C| = O(\mu \log \mu)$, $|\mu \log \mu| = O(\mu \log \mu)$, $\frac{X^2}{2\mu} = O(\mu)$, and $\frac{|X|^3}{6\mu^2} = O(\mu)$, it follows that $|R_4| = O(\mu \log \mu)$. Therefore,
$$
\mathbb{E}\left[|R_4|\mathbf{1}_{G^c}\right]
= O(\mu \log \mu) \, \mathbb{P}(C < \mu/2)
= O(\mu \log \mu \, e^{-\mu/8})
= o(1/\mu).
$$
Thus, we obtain $\mathbb{E}[R_4] = O(1/\mu)$. Taking expectations in \eqref{re2} gives the second equation of the Proposition.
\end{proof}

\section*{Declarations}

\noindent\textbf{Funding} The author is supported by Investigadores por México SECIHTI.

\noindent\textbf{Conflict of interest} The author declares no conflict of interest.

\noindent\textbf{Data availability} The computational data analyzed in this study are derived from the publicly available OEIS sequences A004490 and A073751. No new primary dataset was generated.

\noindent\textbf{AI Usage Statement} During the preparation of this work, the author used two generative AI tools:

\begin{itemize}
    \item \textbf{DeepSeek} was used exclusively for language polishing, grammar correction, and stylistic improvements of the English text. 
    
    \item \textbf{GPT} (OpenAI) was used as a computational assistant in the following specific tasks:
    \begin{itemize}
    \item Optimization, debugging, and computational implementation assistance of the author's SageMath and Python code for processing the first 10,000 colossally abundant numbers generated by the prime sequence A073751 (assuming the Alaoglu-Erd\"os conjecture).
    \item Methodological suggestions during the development of Section \ref{Secto}.
    \item Assistance with the visualization and graphical rendering of Figure 1.
\end{itemize}
\end{itemize}

The author assumes full responsibility for the accuracy, integrity, and originality of the final manuscript. All AI-generated suggestions were critically reviewed, independently verified, and modified as necessary.


\begin{thebibliography}{2}
    \bibitem{AlEr44} L. Alaoglu and P. Erd\"os, \emph{On highly Composite and Similar Numbers}. Transactions of the American Mathematical Society, Vol. 56, No. 3 (1944), pp. 448-469.
    \bibitem{AS66} M. Abramowitz and I. A. Stegun, \emph{Handbook of Mathematical Functions with Formulas, Graphs, and Mathematical Tables.} National Bureau of Standards. Applied Mathematics Series 55, Issue 1964. 
    \bibitem{ABT00} R. Arratia, A.D. Barbour and S. Tavar\'e, \emph{Logarithmic combinatorial Structures: a Probabilistic Approach}, Eur. Math. Soc., Monographs in Mathematics (2003). 


    \bibitem{Br66} de Bruijn, N. G. \emph{On the number of positive integers $\leq  x$ and free of prime factors $> y$, II}. Indagationes Mathematicae. 28: 239-247 (1966).

    \bibitem{Di30} K. Dickman, On the frequency of numbers containing prime factors of a certain
relative magnitude. Ark. Mat. Astr. Fys. 22, (1930), A10, 1-14.
    
    \bibitem{Da25} J. D. Espinosa ,\emph{A beautiful equivalence of the Riemann Hypothesis}\\ https://zenodo.org/records/18251515 (2025).
    \bibitem{Da26} J. D. Espinosa, \emph{The Assembly Theorem}\\ https://zenodo.org/records/18121760 (2026).
    \bibitem{GA26} R. G\'omez-A\'iza, \emph{Young-lattice diagonals and a doubly graded multiple-zeta decomposition of $e^\gamma$} https://arxiv.org/abs/2608.22561

    
    \bibitem{Lag02} J. C. Lagarias, \emph{An Elementary Problem Equivalent to the Riemann Hypothesis}, the American Mathematical Monthly, Vol. 109, No. 6 (2002), pp. 534-543.
    \bibitem{Mac95} I.G. Macdonald, \emph{Symmetric Functions and Hall Polynomials}, Oxford Math. Monographs, Second Edition (1995). 
    \bibitem{Rie59} G. F. B. Riemann, \emph{Ueber die Anzahl der Primzahlen unter einer gegebenen Gr\"osse}, Monatsber. Akad. Berlin (1859), 671–680.
    \bibitem{Ro84} G. Robin, \emph{Grandes valeurs de la fonction somme des diviseurs et hypoth\`ese de Riemann}, J. Math. Pures Appl. 63 (1984), no. 2, 187–213.
    \bibitem{Seg26} C. Segovia \emph{The Riemann Hypothesis in Oaxaca}.
    https://arxiv.org/pdf/2601.22413v2

    \bibitem{Te06} G. Tenenbaum, \emph{Integers with a large friable component}, Acta Arithmetica 124.3 (2006).
    https://tenenb.perso.math.cnrs.fr/PPP/LargeFriableDiv.pdf
\end{thebibliography}
\end{document}